\documentclass[11pt,english]{amsart}
\usepackage[T1]{fontenc}
\usepackage[latin9]{inputenc}
\usepackage{babel}
\usepackage{amsthm}
\usepackage{amssymb}
\usepackage[unicode=true,
 bookmarks=true,bookmarksnumbered=false,bookmarksopen=false,
 breaklinks=false,pdfborder={0 0 1},backref=false,colorlinks=false]
 {hyperref}
\hypersetup{pdftitle={Essential normality and the decomposability of homogeneous submodules},
 pdfauthor={Matthew Kennedy}}

\makeatletter
\numberwithin{equation}{section}
\numberwithin{figure}{section}
\theoremstyle{plain}
\newtheorem{thm}{\protect\theoremname}[section]
  \theoremstyle{definition}
  \newtheorem{defn}[thm]{\protect\definitionname}
  \theoremstyle{remark}
  \newtheorem{rem}[thm]{\protect\remarkname}
  \theoremstyle{plain}
  \newtheorem{lem}[thm]{\protect\lemmaname}
  \theoremstyle{plain}
  \newtheorem{prop}[thm]{\protect\propositionname}
  \theoremstyle{definition}
  \newtheorem{example}[thm]{\protect\examplename}

\usepackage[T1]{fontenc}
\usepackage{ae,aecompl}

\usepackage{microtype}

\makeatother

  \providecommand{\definitionname}{Definition}
  \providecommand{\examplename}{Example}
  \providecommand{\lemmaname}{Lemma}
  \providecommand{\propositionname}{Proposition}
  \providecommand{\remarkname}{Remark}
\providecommand{\theoremname}{Theorem}

\begin{document}

\title[Essential normality and decomposability]{Essential normality and the decomposability of homogeneous submodules}

\author{Matthew Kennedy}

\address{School of Mathematics and Statistics, Carleton University, 1125 Colonel
By Drive, Ottawa, Ontario K1S 5B6, Canada}

\email{mkennedy@math.carleton.ca}
\begin{abstract}
We establish the essential normality of a large new class of homogeneous
submodules of the finite rank $d$-shift Hilbert module. The main
idea is a notion of essential decomposability that determines when
a submodule can be decomposed into the algebraic sum of essentially
normal submodules. We prove that every essentially decomposable submodule
is essentially normal, and introduce methods for establishing that
a submodule is essentially decomposable. It turns out that many submodules
have this property. We prove that many of the submodules considered
by other authors are essentially decomposable, and in addition establish
the essential decomposability of a large new class of homogeneous
submodules. Our results support Arveson's conjecture that every homogeneous
submodule of the finite rank $d$-shift Hilbert module is essentially
normal.
\end{abstract}


\maketitle

\section{Introduction}

In this paper we establish new results in higher-dimensional operator
theory that support Arveson's conjecture of a corresponence between
algebraic varieties and C{*}-algebras of essentially normal operators.
Specifically, we prove the essential normality of a large new class
of homogeneous submodules of the finite rank $d$-shift Hilbert module
introduced by Arveson in \cite{Arv98}. Our work provides a new perspective
on Arveson's conjecture that every homogeneous submodule is essentially
normal.

For fixed $d\geq1$, let $\mathbb{C}[z]=\mathbb{C}[z_{1},\ldots,z_{d}]$
denote the algebra of complex polynomials in $d$ variables. With
the introduction of an appropriate inner product, $\mathbb{C}[z]$
can be completed to a space of analytic functions on the complex unit
ball called the Drury-Arveson space, which we denote by $H_{d}^{2}$.
The coordinate multiplication operators $M_{z_{1}},\ldots,M_{z_{d}}$,
defined on $\mathbb{C}[z]$ by
\[
\left(M_{z_{i}}p\right)\left(z_{1},\ldots,z_{d}\right)=z_{i}p\left(z_{1},\ldots,z_{d}\right),\quad p\in\mathbb{C}[z],\ 1\leq i\leq d,
\]
extend to bounded linear operators on $H_{d}^{2}$, and $\left(M_{z_{1}},\dots,M_{z_{d}}\right)$
forms a contractive $d$-tuple of operators called the $d$-shift.
The space $H_{d}^{2}$ can be naturally viewed as a module over $\mathbb{C}[z]$,
with the module action given by
\[
pf=p\left(M_{z_{1}},\ldots,M_{z_{d}}\right)f,\quad p\in\mathbb{C}[z],\ f\in H_{d}^{2}.
\]
Endowed with this module action, $H_{d}^{2}$ is called the $d$-shift
Hilbert module.

The $d$-shift and the $d$-shift Hilbert module $H_{d}^{2}$ are
of fundamental importance in multivariable operator theory, and they
have received a great deal of attention in recent years (see for example
\cite{Arv98}, \cite{Arv00}, \cite{Arv02}, \cite{Arv05}, \cite{Arv07},
\cite{DRS11}, \cite{DRS12}, \cite{Dou06}, \cite{GW08}, \cite{Esc11},
\cite{Sha11}).

Let $M$ be a submodule of $H_{d}^{2}$, so that $M$ is an invariant
subspace for each coordinate operator $M_{z_{1}},\ldots,M_{z_{d}}$.
If we identify the quotient space $H_{d}^{2}/M$ with the orthogonal
complement $M^{\perp}$, then we obtain a $d$-tuple of quotient operators
$\left(T_{1},\ldots,T_{d}\right)$ by compressing the $d$-shift $\left(M_{z_{1}},\ldots,M_{z_{d}}\right)$
to $H_{d}^{2}/M$. Since the operators $M_{z_{1}},\ldots,M_{z_{d}}$
commute, the operators $T_{1},\ldots,T_{d}$ also commute, and in
fact, every commuting contractive $d$-tuple of operators can be realized
as a quotient of the $d$-shift in precisely this way, provided that
one is willing to increase the multiplicity and consider vector-valued
functions (see for example \cite{Arv98}).

The quotient module $H_{d}^{2}/M$ is said to be \emph{$p$-essentially
normal} if the self-commutators 
\[
T_{i}^{*}T_{j}-T_{j}T_{i}^{*},\quad1\leq i,j\leq d
\]
belong to the Schatten $p$-class $\mathcal{L}^{p}$ for $1\leq p\leq\infty$
(where $\mathcal{L}^{\infty}$ denotes the ideal of compact operators
$\mathcal{K}$). We also obtain a $d$-tuple of operators $(S_{1},\ldots,S_{d})$
by restricting the elements in the $d$-shift $(M_{z_{1}},\ldots,M_{z_{d}})$
to $M$, and the module $M$ is similarly said to be $p$-essentially
normal if the self-commutators
\[
S_{i}^{*}S_{j}-S_{j}S_{i}^{*},\quad1\leq i,j\leq d
\]
belong to $\mathcal{L}^{p}$ for $1\leq p\leq\infty$. In fact, it
turns out that these notions of essential normality are equivalent
for $p>d$, since the submodule $M$ is $p$-essentially normal if
and only if the quotient module $H_{d}^{2}/M$ is $p$-essentially
normal (see for example \cite{Arv05}). 

The purpose of this paper is to consider the essential normality of
submodules of the $d$-shift Hilbert module $H_{d}^{2}$, and more
generally, the essential normality of submodules of the finite rank
$d$-shift Hilbert module $H_{d}^{2}\otimes\mathbb{C}^{r}$, obtained
by tensoring $H_{d}^{2}$ with $\mathbb{C}^{r}$, for a positive integer
$r\geq1$. Note that elements in $H_{d}^{2}\otimes\mathbb{C}^{r}$
can be viewed as analytic vector-valued functions on the complex unit
ball.

Arveson observed in \cite{Arv02} that $H_{d}^{2}\otimes\mathbb{C}^{r}$
is itself $p$-essentially normal for every $p>d$, and motivated
by applications to multivariable Fredholm Theory, he asked whether
every homogeneous submodule of the finite rank $d$-shift Hilbert
module, i.e. every submodule generated by homogeneous polynomials,
is $p$-essentially normal for every $p>d$.

In \cite{Arv05}, Arveson conjectured that this question should have
an affirmative answer, and established the truth of his conjecture
for submodules of $H_{d}^{2}\otimes\mathbb{C}^{r}$ generated by monomials,
i.e. polynomials of the form $z_{1}^{\alpha_{1}}\cdots z_{d}^{\alpha_{d}}\otimes\xi$
for $\alpha=\left(\alpha_{1},\ldots,\alpha_{d}\right)$ in $\mathbb{N}_{0}^{d}$
and $\xi$ in $\mathbb{C}$. More recently, in \cite{GW08}, Guo and
Wang proved that every submodule of $H_{d}^{2}\otimes\mathbb{C}^{r}$
generated by a single homogeneous polynomial is $p$-essentially normal
for every $p>d$. Additionally, they proved that for $d\leq3$, every
homogeneous submodule of $H_{d}^{2}\otimes\mathbb{C}^{r}$ is $p$-essentially
normal for every $p>d$. However, none of these results apply to homogeneous
submodules of $H_{d}^{2}\otimes\mathbb{C}^{r}$ when $d\geq4$ and
the submodule is generated by two or more non-monomials. 

In this paper, we establish Arveson's conjecture for a large new class
of homogeneous submodules of the finite-rank $d$-shift Hilbert module.
This class includes homogeneous submodules of $H_{d}^{2}\otimes\mathbb{C}^{r}$,
with $d$ arbitrarily large, that are generated by an arbitrary number
of non-monomials. For example, we obtain the following result.
\begin{thm}
Let $F_{1},\ldots,F_{n}$ be sets of homogeneous polynomials in $\mathbb{C}[z_{1},\ldots,z_{d}]$.
Suppose that there are disjoint subsets $Z_{1},\ldots,Z_{n}$ of $\{z_{1},\ldots,z_{d}\}$,
each of size at most $2$, such that
\[
F_{i}\subseteq\mathbb{C}[Z_{i}],\quad1\leq i\leq n.
\]
Let $X_{1},\ldots,X_{n}$ be arbitrary sets of vectors in $\mathbb{C}^{r}$.
Then the $H_{d}^{2}\otimes\mathbb{C}^{r}$ submodule generated by
the set of vector-valued polynomials
\[
\{p\otimes\xi\mid p\in F_{i},\ \xi\in X_{i},\ 1\leq i\leq n\}
\]
is $p$-essentially normal for every $p>d$. 
\end{thm}

We obtain Theorem 1.1 as a special case of the following more broadly
applicable result.
\begin{thm}
\label{thm:big-class-ess-norm-1}Let $F_{1},\ldots,F_{n}$ be sets
of polynomials in $\mathbb{C}[z_{1},\ldots,z_{d}]$ that each generate
$p$-essentially normal submodules of $H_{d}^{2}$. Suppose that there
are disjoint subsets $Z_{1},\ldots,Z_{n}$ of $\{z_{1},\ldots,z_{d}\}$,
such that
\[
F_{i}\subseteq\mathbb{C}[Z_{i}],\quad1\leq i\leq n.
\]
Let $X_{1},\ldots,X_{n}$ be arbitrary sets of vectors in $\mathbb{C}^{r}$.
Then the $H_{d}^{2}\otimes\mathbb{C}^{r}$ submodule generated by
the set of vector-valued polynomials
\[
\{p\otimes\xi\mid p\in F_{i},\ \xi\in X_{i},\ 1\leq i\leq n\}
\]
is $p$-essentially normal for every $p>d$.
\end{thm}

More generally, we obtain Theorem 1.2 as an application of a new method
for establishing the essential normality of a submodule of $H_{d}^{2}\otimes\mathbb{C}^{r}$.
First, we observe that if $M_{1},\ldots,M_{n}$ are $p$-essentially
normal submodules of $H_{d}^{2}\otimes\mathbb{C}^{r}$ with the property
that the algebraic sum $M_{1}+\ldots+M_{n}$ is closed, then the $H_{d}^{2}\otimes\mathbb{C}^{r}$
submodule generated by $M_{1},\ldots,M_{n}$ is also $p$-essentially
normal. Reversing this argument tells us that if we want to prove
the $p$-essential normality of a $H_{d}^{2}\otimes\mathbb{C}^{r}$
submodule $M$, then we should try to obtain a decomposition of $M$
as $M=M_{1}+\ldots+M_{n}$, where $M_{1},\ldots,M_{n}$ are $p$-essentially
normal submodules of $H_{d}^{2}\otimes\mathbb{C}^{r}$.

Note that by Guo and Wang's result on the essential normality of submodules
generated by a single homogeneous polynomial, every homogeneous submodule
of $H_{d}^{2}\otimes\mathbb{C}^{r}$ can be written as a closed sum
of essentially normal submodules. Therefore, the main difficulty is
understanding when, if ever, the algebraic sum $M_{1}+\ldots+M_{n}$
is closed. While this problem seems quite difficult in general, we
prove below that this sum is closed in many interesting cases.

Our results also imply the essential normality of submodules that
have been considered by other authors. For example, our results imply
the essential decomposability of submodules of $H_{d}^{2}\otimes\mathbb{C}^{r}$
generated by monomials, and so we obtain a new proof of Arveson's
main result in \cite{Arv05}.

In addition to this introduction, this paper has five other sections.
In Section 2, we provide a brief review of the basic background material.
In Section 3, we introduce the notion of essential decomposability,
and relate it to Shalit's stable division property from \cite{Sha11}.
In Section 4, we introduce a notion of perpendicularity for a family
of submodules that implies essential decomposability. In Section 5,
we establish the the main results on essential normality.

\section{Preliminaries}

\subsection{\label{sub:prelim-drury-areson}The $d$-shift Hilbert module}

For fixed $d\geq1$, let $\mathbb{C}[z]=\mathbb{C}[z_{1},\ldots,z_{d}]$
denote the algebra of complex polynomials in $d$ variables. For monomials
in $\mathbb{C}[z]$, it is convenient to use the multi-index notation
\[
z^{\alpha}=z_{1}^{\alpha_{1}}\cdots z_{d}^{\alpha_{d}},\quad\alpha=(\alpha_{1},\ldots,\alpha_{d})\in\mathbb{N}_{0}^{d}.
\]
The \emph{Drury-Arveson space} $H_{d}^{2}$ is the completion of $\mathbb{C}[z]$
with respect to the inner product $\left\langle \cdot,\cdot\right\rangle $,
defined on monomials by
\[
\langle z^{\alpha},z^{\beta}\rangle=\delta_{\alpha\beta}\frac{\alpha!}{|\alpha|!},\quad\alpha,\beta\in\mathbb{N}_{0}^{d},
\]
where we have written $\alpha!=\alpha_{1}!\cdots\alpha_{d}!$ and
$|\alpha|=\alpha_{1}+\ldots+\alpha_{d}$ for $\alpha=(\alpha_{1},\ldots,\alpha_{d})$
in $\mathbb{N}_{0}^{d}$.

Let $M_{z_{1}},\ldots,M_{z_{d}}$ denote the coordinate multiplication
operators on $\mathbb{C}[z]$ corresponding to the variables $z_{1},\ldots,z_{d}$
respectively,
\[
M_{z_{i}}p=z_{i}p,\quad p\in\mathbb{C}[z],\ 1\leq i\leq d.
\]
Then these operators extend to bounded linear operators on $H_{d}^{2}$,
and the $d$-tuple $(M_{z_{1}},\ldots,M_{z_{d}})$ is called the \emph{$d$-shift}.

The elements in $H_{d}^{2}$ can be identified with analytic functions
on the complex unit ball, and $H_{d}^{2}$ can be viewed as a Hilbert
module over the algebra of polynomials $\mathbb{C}[z]$, with the
module action defined by
\[
pf=p(M_{z_{1}},\ldots,M_{z_{d}})f,\quad p\in\mathbb{C}[z],\ f\in H_{d}^{2}.
\]
Endowed with this module action, $H_{d}^{2}$ is called the \emph{$d$-shift
Hilbert module}. The importance of this construction in multivariable
operator theory was recognized by Arveson in his comprehensive treatment
\cite{Arv98}.

Let $N$ denote the number operator, the unbounded self-adjoint operator
defined on monomials in $\mathbb{C}[z]$ by
\[
Nz^{\alpha}=|\alpha|z^{\alpha},\quad\alpha\in\mathbb{N}_{0}^{d},
\]
and extended to polynomials in $\mathbb{C}[z]$ by linearity. Then
the operator $\left(N+1\right)^{-1}$ extends to a bounded operator
on $H_{d}^{2}$. Let $\partial_{1},\ldots,\partial_{d}$ denote the
operators that act on $\mathbb{C}[z]$ by partial differentiation
with respect to the variables $z_{1},\ldots,z_{d}$ respectively.
Then restricted to $\mathbb{C}[z]$, we can write 
\[
M_{z_{i}}^{*}=\left(N+1\right)^{-1}\partial_{i},\quad1\leq i\leq d.
\]
The fact that the adjoint operators take this form is one reason for
the importance of the $d$-shift (see \cite{Arv98} for details).

More generally, for a polynomial $p$ in $\mathbb{C}[z]$, let $M_{p}$
denote the operator on $H_{d}^{2}$ corresponding to multiplication
by $p$,
\[
M_{p}f=pf,\quad f\in H_{d}^{2}.
\]

\subsection{\label{sub:finite-rank-d-shift}The finite rank $d$-shift Hilbert
module}

We will need to consider a higher multiplicity version of $H_{d}^{2}$.
For fixed $r\geq1$, the\emph{ $d$-shift Hilbert module of rank $r$,}
$H_{d}^{2}\otimes\mathbb{C}^{r}$, is the Hilbert space tensor product
of $H_{d}^{2}$ with the $r$-dimensional Hilbert space $\mathbb{C}^{r}$.
Note that we could also realize $H_{d}^{2}\otimes\mathbb{C}^{r}$
as the completion of the algebraic tensor product $\mathbb{C}[z]\otimes\mathbb{C}^{r}$.
We will follow \cite{Arv07} and write $r\mathbb{C}[z]$ and $rH_{d}^{2}$
for $\mathbb{C}[z]\otimes\mathbb{C}^{r}$ and $H_{d}^{2}\otimes\mathbb{C}^{r}$
respectively.

Since the meaning will always be clear from the context, it will be
convenient to also let $M_{z_{1}},\ldots,M_{z_{d}}$ denote the coordinate
multiplication operators on $rH_{d}^{2}$,
\[
M_{z_{i}}\left(f\otimes\xi\right)=M_{z_{i}}f\otimes\xi,\quad f\in H_{d}^{2},\ \xi\in\mathbb{C}^{r},\ 1\leq i\leq d.
\]
Note that cordinate multiplication operators on $rH_{d}^{2}$ can
also be realized as the tensor product of the coordinate multiplication
operators on $H_{d}^{2}$ with the identity operator on $\mathbb{C}^{r}$.
The $d$-tuple $(M_{z_{1}},\ldots,M_{z_{d}})$ is called the \emph{$d$-shift
of rank $r$}. 

The elements in $rH_{d}^{2}$ can be viewed as vector-valued analytic
functions on the complex unit ball, and $rH_{d}^{2}$ can also be
viewed as a Hilbert module over the algebra of polynomials $\mathbb{C}[z]$.
In this case, the module action is defined on the elementary tensors
in $rH_{d}^{2}$ by
\[
p\left(f\otimes\xi\right)=p\left(M_{z_{1}},\ldots,M_{z_{d}}\right)\left(f\otimes\xi\right),\quad p\in\mathbb{C}[z],\ f\in H_{d}^{2},\ \xi\in\mathbb{C}^{r},
\]
and extended to all of $rH_{d}^{2}$ by linearity.

\subsection{Essential normality}

Let $N$ be a submodule of $rH_{d}^{2}$. Then $N$ is invariant for
the coordinate multiplication operators $M_{z_{1}},\ldots,M_{z_{d}}$
on $rH_{d}^{2}$, so we can consider the corresponding restrictions
$S_{1},\ldots,S_{d}$ to $N$. The submodule $N$ is said to be \emph{$p$-essentially
normal} if the self-commutators
\[
S_{i}^{*}S_{j}-S_{j}S_{i}^{*},\quad1\leq i,j\leq d,
\]
belong to the Schatten $p$-class $\mathcal{L}^{p}$ for $1\leq p\leq\infty$.
If $p=\infty$, then we will say that $N$ is \emph{essentially normal}.

If we identify the quotient space $rH_{d}^{2}/N$ with the orthogonal
complement $N^{\perp},$ then we can also consider the compressions
$T_{1},\ldots,T_{d}$ of $M_{z_{1}},\ldots,M_{z_{d}}$ respectively
to $rH_{d}^{2}/N$. The quotient module $rH_{d}^{2}/N$ is similarly
said to be \emph{$p$-essentially normal} if the self-commutators
\[
T_{i}^{*}T_{j}-T_{j}T_{i}^{*},\quad1\leq i,j\leq d,
\]
belong to the Schatten class $\mathcal{L}^{p}$ for $1\leq p\leq\infty$.

For a submodule $N$, let  $P_{N}$ denote the projection onto $N$.
We will require the following result of Arveson, which is Theorem
4.3 of \cite{Arv05}.
\begin{thm}
[Arveson]\label{thm:arveson-ess-norm-equiv-conditions}Let $N$ be
a submodule of $rH_{d}^{2}$. Then for every $p>d$, the following
are equivalent:
\begin{enumerate}
\item $N$ is $p$-essentially normal,
\item $rH_{d}^{2}/N$ is $p$-essentially normal,
\item For $1\leq i\leq d$, the commutators $M_{z_{i}}P_{N}-P_{N}M_{z_{i}}$
belong to the Schatten $p$-class $\mathcal{L}^{2p}$.
\end{enumerate}
\end{thm}
In our work, we will mostly use condition (3) of Theorem \ref{thm:arveson-ess-norm-equiv-conditions}.

\section{\label{sec:ess-decomp}Essential decomposability}

\subsection{Essential decomposability}

In this section, we will consider a notion of decomposability for
a submodule that implies essential normality. Let $N_{1},\ldots,N_{n}$
be submodules of $rH_{d}^{2}$. Then we will write $N_{1}+\ldots+N_{n}$
for the (not necessarily closed) algebraic sum
\[
N_{1}+\ldots+N_{n}=\left\{ x_{1}+\ldots+x_{n}\mid x_{i}\in N_{i}\ \mbox{for}\ 1\leq i\leq n\right\} .
\]

\begin{defn}
Let $N$ be a submodule of $rH_{d}^{2}$. Then $N$ is said to be
\emph{$p$-essentially decomposable} if there are $p$-essentially
normal $rH_{d}^{2}$ submodules $N_{1},\ldots,N_{n}$ such that
\[
N=N_{1}+\ldots+N_{n}.
\]
If $p=\infty$, then $N$ is said to be \emph{essentially decomposable}.\end{defn}
\begin{rem}
Note that if $N$ is a $p$-essentially normal submodule of $rH_{d}^{2}$,
then it is trivially $p$-essentially decomposable.
\end{rem}

\begin{thm}
\label{thm:ess-decomp-implies-ess-norm}Every $p$-essentially decomposable
submodule of $rH_{d}^{2}$ is $p$-essentially normal for $p>d$.\end{thm}
\begin{proof}
Let $N$ be a $p$-essentially decomposable submodule of $rH_{d}^{2}$
with decomposition $N=N_{1}+\ldots+N_{n}$, where $N_{1},\ldots,N_{n}$
are $p$-essentially normal $rH_{d}^{2}$ submodules. Let $M=N_{1}\oplus\ldots\oplus N_{n}$.
Then $M$ is a submodule of $nrH_{d}^{2}=(rH_{d}^{2})^{n}$, and the
$p$-essential normality of $N_{1},\ldots,N_{n}$ implies the $p$-essential
normality of $M$. Define $L:(rH_{d}^{2})^{n}\to rH_{d}^{2}$ by
\[
L(x_{1},\ldots,x_{n})=x_{1}+\ldots+x_{n},\quad(x_{1},\ldots,x_{n})\in(rH_{d}^{2})^{n},
\]
Then $L(M_{z_{i}}^{(n)})^{*}=M_{z_{i}}^{*}L$ for each $i$, where
$M_{z_{i}}^{(n)}$ denotes the coordinate multiplication operator
on $(rH_{d}^{2})^{n}$. In particular, the restriction of $L$ to
$M$ is a $2p$-morphism from $(rH_{d}^{2})^{n}$ to $rH_{d}^{2}$,
in the sense of Definition 4.3 of \cite{Arv07}. Furthermore, $L(M)=N$
is closed. Since $M$ is $p$-essentially normal, it follows from
Theorem 4.4 of \cite{Arv07} that $N$ is also $p$-essentially normal.
\end{proof}

\begin{rem}
A direct proof of Theorem \ref{thm:ess-decomp-implies-ess-norm} that
avoids the notion of a $p$-morphism can be given by emulating the
first part of the proof of Theorem 4.4 of \cite{Arv07}.
\end{rem}

We will also require the following lemma, which can be proved by a
simple modification of the proof of Corollary 3 of \cite{FW71}.
\begin{lem}
\label{lem:vec-sum-closed-iff-proj-sum-closed}Let $N_{1},\ldots,N_{n}$
be submodules of $rH_{d}^{2}$. Then the algebraic sum $N_{1}+\ldots+N_{n}$
is closed if and only if the range of the operator $P_{N_{1}}+\ldots+P_{N_{n}}$
is closed.
\end{lem}

\begin{rem}
A classical result of Friedrichs \cite{Fri37} implies that the algebraic
sum of two subspaces $N_{1}$ and $N_{2}$ is closed if and only if
the (Friedrichs) angle between them is positive. Recently in \cite{BGM10},
Badea, Grivaux and Muller established an analogue of Friedrichs' result
for an arbitrary number of subspaces $N_{1},\ldots,N_{n}$, by considering
a generalized notion of angle. Although we do not require their results
in the present paper, we believe that similar ideas may eventually
prove useful in resolving Arveson's conjecture.
\end{rem}

\subsection{\label{sub:stable-division}Stable division}

The stable division property was introduced by Shalit in \cite{Sha11},
in connection with Arveson's conjecture. However, the notion of stable
division is notion is also of independent interest, since it concerns
the numerical stability of multivariable polynomial division.
\begin{defn}
An $rH_{d}^{2}$ submodule $N$ is said to have the \emph{stable division
property} if there is a family of homogeneous polynomials $\{p_{1},\ldots,p_{n}\}$
generating $N$ and a constant $C\geq0$ such that, for any polynomial
$p$ in $N$, there are polynomials $q_{1},\ldots,q_{n}$ in $\mathbb{C}[z]$
satisfying $p=q_{1}p_{1}+\ldots+q_{n}p_{n}$ and
\[
\|q_{1}p_{1}\|+\ldots+\|q_{n}p_{n}\|\leq C\|p\|.
\]
The family $\{p_{1},\ldots,p_{n}\}$ is said to be a \emph{stable
generating set} for $N$.
\end{defn}

The next result was discovered at the suggestion of  Shalit. It establishes
a connection between the ideas in this paper and his work on the stable
division property in \cite{Sha11}.
\begin{thm}
\label{thm:stable-div-equiv-friedrichs-num}Let $p_{1},\ldots,p_{n}$
be homogeneous polynomials in $r\mathbb{C}[z]$, and let $N_{1},\ldots,N_{n}$
denote the corresponding $rH_{d}^{2}$ submodules they generate. Then
$\{p_{1},\ldots,p_{n}\}$ is a stable generating set if and only if
the algebraic sum $N_{1}+\ldots+N_{n}$ is closed.\end{thm}
\begin{proof}
Let $N=\overline{N_{1}+\ldots+N_{n}}$. Suppose first that $N_{1}+\ldots+N_{n}$
is closed. Then the operator $T:N_{1}\oplus\ldots\oplus N_{n}\to N$,
defined by
\[
T(x_{1},\ldots,x_{n})=x_{1}+\ldots+x_{n},\quad(x_{1},\ldots,x_{n})\in N_{1}\oplus\ldots\oplus N_{n},
\]
has closed range $N$. Hence there is a constant $C\geq0$ such that
for any $f$ in $N$, there are $f_{i}$ in $N_{i}$ satisfying $f=f_{1}+\ldots+f_{n}$
and
\[
\|f_{1}\|+\ldots+\|f_{n}\|\leq C\|f\|.
\]
If $f=p$ is a homogeneous polynomial, then we can replace each $f_{i}$
with the homogeneous polynomial $q_{i}'$ obtained by projecting $f_{i}$
onto the homogeneous component of $rH_{d}^{2}$ containing $p$. Since
$N_{i}$ is generated by the homogeneous polynomial $p_{i}$, it is
left invariant by this projection. Hence $q_{i}'$ still belongs to
$N_{i}$, and we can write $q_{i}'=q_{i}p_{i}$ for some polynomial
$q_{i}$ in $\mathbb{C}[z]$. Since an arbitrary polynomial can be
written as an orthogonal sum of homogeneous polynomials of different
degrees, it follows that $\{p_{1},\ldots,p_{n}\}$ is a stable generating
set for $N$.

Conversely, suppose $\{p_{1},\ldots,p_{n}\}$ is a stable generating
set for $N$. Fix $f$ in $N$, and let $(s_{k})_{k=1}^{\infty}$
be a sequence of polynomials in $N$ converging to $f$. By the stable
division property, there is a constant $C\geq0$ such that, for each
$k\geq1$, there are polynomials $q_{k,i}$ in $\mathbb{C}[z]$ satisfying
$s_{k}=q_{k,1}p_{1}+\ldots+q_{k,n}p_{n}$ and
\[
\|q_{k,1}p_{1}\|+\ldots+\|q_{k,n}p_{n}\|\leq C\|s_{k}\|.
\]
For each $i$, the sequence $(q_{k,i}p_{i})_{k=1}^{\infty}$ is clearly
bounded, and by passing to a subsequence we can assume that it is
weakly convergent to some $f_{i}$ in $N_{i}$. Then for all $g$
in $rH_{d}^{2}$,
\[
\langle f-(f_{1}+\ldots+f_{n}),g\rangle=\lim_{k\to\infty}\langle s_{k}-(q_{k,1}p_{1}+\ldots+q_{k,n}p_{n}),g\rangle=0,
\]
and it follows that $f=f_{1}+\ldots+f_{k}$. Hence $f$ belongs to
$N_{1}+\ldots+N_{n}$. Since $f$ was arbitrary, we conclude that
$N=N_{1}+\ldots+N_{n}$.
\end{proof}

Shalit proved in Theorem 2.3 of \cite{Sha11} that many families of
homogeneous submodules of $H_{2}^{2}$ have the stable division property.
However, consideration of Shalit's proof reveals that it actually
implies the following stronger result. (For background material on
Groebner bases, see, for example \cite{CLS92}.)
\begin{thm}
[Shalit]\label{thm:shalit-grobner}Let $p_{1},\ldots,p_{n}$ be homogeneous
polynomials in $\mathbb{C}[z]$ such that $\{p_{1},\ldots,p_{n}\}$
is a Groebner basis. Suppose there is a subset $Z$ of $\{z_{1},\ldots,z_{d}\}$,
of size at most $2$, such that $p_{1},\ldots,p_{n}\in\mathbb{C}[Z]$.
Then the family $\{p_{1},\ldots,p_{n}\}$ is a stable generating set.
\end{thm}

Applying Theorem \ref{thm:stable-div-equiv-friedrichs-num} to Theorem
\ref{thm:shalit-grobner}, we obtain the following result.
\begin{prop}
\label{prop:grobner-basis-implies-friedrichs-num}Let $N_{1},\ldots,N_{n}$
be submodules of $H_{d}^{2}$ generated by homogeneous polynomials
$p_{1},\ldots,p_{n}$ respectively in $\mathbb{C}[z]$. Suppose that
$\{p_{1},\ldots,p_{n}\}$ is a Groebner basis, and suppose that there
is a subset $Z$ of $\{z_{1},\ldots,z_{d}\}$, of size at most $2$,
such that $p_{1},\ldots,p_{n}\in\mathbb{C}[Z]$. Then the algebraic
sum $N_{1}+\ldots+N_{n}$ is closed.
\end{prop}

Applying Theorem \ref{thm:stable-div-equiv-friedrichs-num} to an
example from \cite{Sha11} provides an example of two submodules of
$H_{d}^{2}$ with non-closed algebraic sum, and demonstrates that
there is no straightforward generalization of Proposition \ref{prop:grobner-basis-implies-friedrichs-num}
to polynomials in three or more variables.
\begin{example}
\label{ex:cosine-is-one}Let $N_{1}$ and $N_{2}$ denote the $H_{3}^{2}$
submodules generated by the polynomials $p_{1}$ and $p_{2}$ respectively,
where
\begin{align*}
p_{1}\left(z_{1},z_{2},z_{3}\right) & =z_{1}^{2}+z_{2}z_{3}\\
p_{2}\left(z_{1},z_{2},z_{3}\right) & =z_{2}^{2},
\end{align*}
and let $N=\overline{N_{1}+N_{2}}$. In Example 2.6 of \cite{Sha11}
it was shown that the family $\{p_{1},p_{2}\}$ generates $N$, but
is not a stable generating set. Hence by Theorem \ref{thm:stable-div-equiv-friedrichs-num},
the algebraic sum $N_{1}+N_{2}$ is not closed. Since $\{p_{1},p_{2}\}$
is a Groebner basis with respect to the lexicographical monomial ordering,
this also shows that Proposition 3.8 does not generalize to polynomials
in three or more variables.

However, $N$ is essentially decomposable. Indeed, let $K_{1},K_{2},K_{3},K_{4}$
denote the submodules of $H_{3}^{2}$ generated by the polynomials
$q_{1},q_{2},q_{3},q_{4}$ respectively, where
\begin{align*}
q_{1}\left(z_{1},z_{2},z_{3}\right) & =z_{1}^{4}\\
q_{2}\left(z_{1},z_{2},z_{3}\right) & =z_{1}^{2}z_{2}\\
q_{3}\left(z_{1},z_{2},z_{3}\right) & =z_{1}^{2}+z_{2}z_{3}\\
q_{4}\left(z_{1},z_{2},z_{3}\right) & =z_{2}^{2}.
\end{align*}
Then the family $\{q_{1},q_{2},q_{3},q_{4}\}$ is a stable generating
set for $N$. Hence by Theorem \ref{thm:stable-div-equiv-friedrichs-num},
$N=K_{1}+K_{2}+K_{3}+K_{4}$.
\end{example}

We also briefly mention Eschmeier's recent paper \cite{Esc11}, which
introduces a related property of a family of polynomials,  in connection
with Arveson's essential normality conjecture. In fact, as pointed
out in \cite{Sha11}, Eschmeier's property is implied by the stable
division property.

\section{\label{sec:perpendicular-submodules}Perpendicular submodules}

\subsection{Perpendicularity}

In this section we consider a property of a family of submodules of
$rH_{d}^{2}$ that implies the algebraic sum of the submodules is
closed.
\begin{defn}
\label{def:perpendicularity}Let $N_{1},\ldots,N_{n}$ be submodules
of $rH_{d}^{2}$. The family $\{N_{1},\ldots,N_{n}\}$ is \emph{perpendicular}
if
\begin{equation}
N_{i}\cap(N_{i}\cap N_{j})^{\perp}\perp N_{j}\cap(N_{i}\cap N_{j})^{\perp},\quad1\leq i<j\leq n.\label{eq:prop-perpendicularity-condition}
\end{equation}

\end{defn}

\begin{prop}
\label{prop:perpendicular-iff-projections-commuting}Let $N_{1},\ldots,N_{n}$
be submodules of $rH_{d}^{2}$. Then the family $\{N_{1},\ldots,N_{n}\}$
is perpendicular if and only if the projections $P_{N_{1}},\ldots,P_{N_{n}}$
commute.\end{prop}
\begin{proof}
We recall the simple fact that if $P$ and $Q$ are projections with
range $\operatorname{ran}\left(P\right)$ and $\operatorname{ran}\left(Q\right)$
respectively, then $P$ and $Q$ commute if and only if 
\[
\operatorname{ran}\left(P\right)\cap\left(\operatorname{ran}\left(P\right)\cap\operatorname{ran}\left(Q\right)\right)^{\perp}\perp\operatorname{ran}\left(Q\right)\cap\left(\operatorname{ran}\left(P\right)\cap\operatorname{ran}\left(Q\right)\right)^{\perp}.
\]
Therefore, the result follows immediately from Definition \ref{def:perpendicularity}.
\end{proof}

Ken Davidson pointed out that the following lemma is a well known
result from the theory of CSL (commutative subspace lattice) algebras.
\begin{lem}
\label{lem:perp-family-lattice-perp}Let $\{N_{1},\ldots,N_{n}\}$
be a perpendicular family of submodules of $rH_{d}^{2}$, and let
$K_{1},\ldots,K_{m}$ be subspaces contained in the subspace lattice
generated by $N_{1},\ldots,N_{n}$. Then $\{K_{1},\ldots,K_{m}\}$
is also a perpendicular family of submodules of $rH_{d}^{2}$.\end{lem}
\begin{proof}
The projections $P_{K_{1}},\ldots,P_{K_{n}}$ are contained in the
von Neumann algebra generated by the projections $P_{N_{1}},\ldots,P_{N_{n}}$,
and by Proposition \ref{prop:perpendicular-iff-projections-commuting},
this von Neumann algebra is commutative. In particular, the projections
$P_{K_{1}},\ldots,P_{K_{n}}$ commute, and another application of
Proposition \ref{prop:perpendicular-iff-projections-commuting} implies
that the family $\{K_{1},\ldots,K_{m}\}$ is perpendicular.
\end{proof}

\subsection{\label{sub:criteria-for-perpendicularity}Criteria for perpendicularity}

In this section, we will consider criteria for a family of submodules
to be perpendicular.
\begin{lem}
\label{lem:generators-comm-decomp}Let $N_{1},\ldots,N_{n}$ be submodules
of $H_{d}^{2}$, and for $1\leq i\leq n$, let $p_{i1},\ldots,p_{im_{i}}$
be polynomials that generate $N_{i}$. If the operators 
\[
M_{p_{i1}}M_{p_{i1}}^{*}+\ldots+M_{p_{im_{i}}}M_{p_{im_{i}}}^{*},\quad1\leq i\leq n.
\]
commute, then the family $\{N_{1},\ldots,N_{n}\}$ is perpendicular.\end{lem}
\begin{proof}
For $1\leq i,j\leq n$, $P_{N_{i}}$ and $P_{N_{j}}$ are the range
projections of the operators $M_{p_{i1}}M_{p_{i1}}^{*}+\ldots+M_{p_{im_{i}}}M_{p_{im_{i}}}^{*}$
and $M_{p_{j1}}M_{p_{j1}}^{*}+\ldots+M_{p_{jm_{j}}}M_{p_{jm_{j}}}^{*}$
respectively. In particular, the projections $P_{N_{i}}$and $P_{N_{j}}$
are contained in the von Neumann algebra generated by $M_{p_{i1}}M_{p_{i1}}^{*}+\ldots+M_{p_{im_{i}}}M_{p_{im_{i}}}^{*}$
and $M_{p_{j1}}M_{p_{j1}}^{*}+\ldots+M_{p_{jm_{j}}}M_{p_{jm_{j}}}^{*}$.
Since these latter operators are self-adjoint, and since they commute,
this von Neumann algebra is commutative, meaning in particular that
the projections $P_{N_{i}}$ and $P_{N_{j}}$ commute. Since $i$
and $j$ were arbitrary, Proposition  \ref{prop:perpendicular-iff-projections-commuting}
implies that the family $\{N_{1},\ldots,N_{n}\}$ is perpendicular. 
\end{proof}

To apply Lemma \ref{lem:generators-comm-decomp}, we will require
an identity of Guo and Wang from \cite{GW08}. Before presenting the
identity, it will be convenient to introduce some special notation
for operators that are related to the number operator $N$ defined
in Section \ref{sub:prelim-drury-areson}. For a function $f:\mathbb{Z}\to\mathbb{Z}$,
let $\left[f\left(N\right)\right]$ denote the (potentially unbounded)
self-adjoint operator defined on monomials in $\mathbb{C}[z]$ by
\[
\left[f\left(N\right)\right]z^{\alpha}=f\left(\left|\alpha\right|\right)z^{\alpha},\quad\alpha\in\mathbb{N}_{0}^{d},
\]
and extended by linearity to polynomials in $\mathbb{C}[z]$. Then,
for example, restricted to $\mathbb{C}[z]$, we can write the adjoints
of the coordinate multiplication operators $M_{z_{1}}^{*},\ldots,M_{z_{d}}^{*}$
on $H_{d}^{2}$ as
\[
M_{z_{i}}^{*}=\left[\frac{1}{N+1}\right]\partial_{i},\quad1\leq i\leq d,
\]
where $\partial_{1},\ldots,\partial_{d}$ denote the operators that
act on $\mathbb{C}[z]$ by partial differentiation in the variable
$z_{1},\ldots,z_{d}$ respectively. If the operator $\left[f\left(N\right)\right]$
happens to extend to a bounded operator on $H_{d}^{2}$, then we will
also write $\left[f\left(N\right)\right]$ for this extension. Recall
that for a polynomial $p$ in $\mathbb{C}[z]$, we write $M_{p}$
to denote the operator on $H_{d}^{2}$ corresponding to multiplication
by $p$. If $p$ is homogeneous of degree $n,$ then it is easy to
check that, restricted to $\mathbb{C}[z]$, we can write
\[
\left[f\left(N\right)\right]M_{p}=M_{p}\left[f\left(N+n\right)\right].
\]
These facts, combined with the general Leibniz rule
\[
\partial^{\alpha}\left(pq\right)=\sum_{\substack{\beta\in\mathbb{N}_{0}^{d}\\
\beta\leq\alpha
}
}\binom{\alpha}{\beta}\left(\partial^{\alpha-\beta}p\right)\left(\partial^{\beta}q\right),\quad\alpha\in\mathbb{N}_{0}^{d},\ p,q\in\mathbb{C}[z],
\]
where $\partial^{\alpha}=\partial_{\alpha_{1}}\cdots\partial_{\alpha_{d}}$,
lead to the following identity of Guo and Wang from \cite{GW08}.
\begin{prop}
[Guo-Wang Identity]\label{prop:guo-wang-ident}Let $p$ and $q$
be homogeneous polynomials in $\mathbb{C}[z]$ of degree $m$ and
$n$ respectively. Then
\[
M_{p}^{*}M_{q}=\sum_{\alpha\in\mathbb{N}_{0}^{d}}\frac{1}{\alpha!}\left[\frac{N!\left(N+m-n\right)!}{\left(N+m\right)!\left(N-n+\left|\alpha\right|\right)!}\right]M_{\partial^{\alpha}q}M_{\partial^{\alpha}p}^{*}.
\]

\end{prop}

We will apply Proposition \ref{prop:guo-wang-ident} to determine
when the hypotheses of Proposition \ref{lem:generators-comm-decomp}
hold.

\begin{lem}
\label{lem:commutator-Mp-Mq}Let $p$ and $q$ be homogeneous polynomials
in $\mathbb{C}[z]$ of degree $m$ and $n$ respectively. Then
\begin{align*}
M_{p} & M_{p}^{*}M_{q}M_{q}^{*}-M_{q}M_{q}^{*}M_{p}M_{p}^{*}\\
 & =\sum_{\alpha\in\mathbb{N}_{0}^{d}\backslash\left\{ 0\right\} }\frac{1}{\alpha!}\left[\tfrac{\left(N-m\right)!\left(N-n\right)!}{N!\left(N-m-n+\left|\alpha\right|\right)!}\right]\left(M_{p}M_{\partial^{\alpha}q}M_{q}^{*}M_{\partial^{\alpha}p}^{*}-M_{\partial^{\alpha}p}M_{q}M_{\partial^{\alpha}q}^{*}M_{p}^{*}\right).
\end{align*}
\end{lem}
\begin{proof}
The Guo-Wang identity from Proposition \ref{prop:guo-wang-ident}
gives
\begin{align*}
M_{p}M_{p}^{*}M_{q}M_{q}^{*} & =\sum_{\alpha\in\mathbb{N}_{0}^{d}}M_{p}\frac{1}{\alpha!}\left[\tfrac{N!\left(N+m-n\right)!}{\left(N+m\right)!\left(N-n+\left|\alpha\right|\right)!}\right]M_{\partial^{\alpha}q}M_{\partial^{\alpha}p}^{*}M_{q}^{*}\\
 & =\sum_{\alpha\in\mathbb{N}_{0}^{d}}\frac{1}{\alpha!}\left[\tfrac{\left(N-m\right)!\left(N-n\right)!}{N!\left(N-m-n+\left|\alpha\right|\right)!}\right]M_{p}M_{\partial^{\alpha}q}M_{q}^{*}M_{\partial^{\alpha}p}^{*},
\end{align*}
and by symmetry this implies
\[
M_{q}M_{q}^{*}M_{p}M_{p}^{*}=\sum_{\alpha\in\mathbb{N}_{0}^{d}}\frac{1}{\alpha!}\left[\tfrac{\left(N-m\right)!\left(N-n\right)!}{N!\left(N-m-n+\left|\alpha\right|\right)!}\right]M_{\partial^{\alpha}p}M_{q}M_{\partial^{\alpha}q}^{*}M_{p}^{*}.
\]
The result now follows by taking the difference of these identities. 
\end{proof}

\begin{lem}
\label{lem:moderate-comm-criterion}Let $p$ and $q$ be homogeneous
polynomials in $\mathbb{C}[z]$ of degree $m$ and $n$ respectively.
Then $M_{p}M_{p}^{*}$ and $M_{q}M_{q}^{*}$ commute if and only if
the operator
\[
\sum_{\alpha\in\mathbb{N}_{0}^{d}\backslash\left\{ 0\right\} }\frac{1}{\alpha!}\left[\frac{\left(N-m\right)!\left(N-n\right)!}{N!\left(N-m-n-\left|\alpha\right|\right)!}\right]M_{p}M_{\partial^{\alpha}q}M_{q}^{*}M_{\partial^{\alpha}p}^{*}
\]
is self-adjoint.\end{lem}
\begin{proof}
This follows immediately from Lemma \ref{lem:commutator-Mp-Mq}, using
the observation that for $\alpha$ in $\mathbb{N}_{0}^{d}$, 
\[
\begin{split}M_{p}M_{\partial^{\alpha}q}M_{q}^{*}M_{\partial^{\alpha}p}^{*} & -M_{\partial^{\alpha}p}M_{q}M_{\partial^{\alpha}q}^{*}M_{p}^{*}\\
 & =M_{p}M_{\partial^{\alpha}q}M_{q}^{*}M_{\partial^{\alpha}p}^{*}-\left(M_{p}M_{\partial^{\alpha}q}M_{q}^{*}M_{\partial^{\alpha}p}^{*}\right)^{*}.
\end{split}
\]

\end{proof}

\begin{lem}
\label{lem:simple-comm-criterion}Let $p$ and $q$ be homogeneous
polynomials in $\mathbb{C}[z]$ of degree $m$ and $n$ respectively.
Then $M_{p}M_{p}^{*}$ and $M_{q}M_{q}^{*}$ commute if each operator
\[
M_{p}M_{\partial^{\alpha}q}M_{q}^{*}M_{\partial^{\alpha}p}^{*},\quad\alpha\in\mathbb{N}_{0}^{d}\backslash\left\{ 0\right\} 
\]
is self-adjoint.\end{lem}
\begin{proof}
This follows immediately from Lemma \ref{lem:moderate-comm-criterion}.
\end{proof}

\begin{lem}
\label{lem:linear-comm-criterion}Let $p$ and $q$ be linear polynomials
in $\mathbb{C}[z]$. Then the operators $M_{p}M_{P}^{*}$ and $M_{q}M_{q}^{*}$
commute if either $p=q$ or $p\perp q$.\end{lem}
\begin{proof}
Write the polynomials $p$ and $q$ as
\[
\begin{aligned}p\left(z_{1},\ldots,z_{d}\right) & =a_{1}z_{1}+\ldots+a_{d}z_{d}\\
q\left(z_{1},\ldots,z_{d}\right) & =b_{1}z_{1}+\ldots+b_{d}z_{d}.
\end{aligned}
\]
Then
\begin{align}
\sum_{\alpha\in\mathbb{N}_{0}^{d}\backslash\left\{ 0\right\} }\frac{1}{\alpha!}\left[\tfrac{\left(N-m\right)!\left(N-n\right)!}{N!\left(N-m-n+\left|\alpha\right|\right)!}\right] & M_{p}M_{\partial^{\alpha}q}M_{q}^{*}M_{\partial^{\alpha}p}^{*}\label{eq:lem-linear-comm-criterion-1}\\
 & =\sum_{i=1}^{d}\left[\tfrac{\left(N-1\right)!\left(N-1\right)!}{N!\left(N+1\right)!}\right]a_{i}\overline{b_{i}}M_{p}M_{q}^{*}\nonumber \\
 & =\left\langle p,q\right\rangle \left[\tfrac{\left(N-1\right)!\left(N-1\right)!}{N!\left(N+1\right)!}\right]M_{p}M_{q}^{*}.\nonumber 
\end{align}
Hence the operator (\ref{eq:lem-linear-comm-criterion-1}) is self-adjoint
if either $p=q$ or $p\perp q$, and the result follows by Lemma \ref{lem:moderate-comm-criterion}.
\end{proof}

\begin{lem}
\label{lem:monomials-comm}Let $z^{\lambda}$ and $z^{\mu}$ be monomials
in $\mathbb{C}[z]$ for $\lambda$ and $\mu$ in $\mathbb{N}_{0}^{d}$.
Then the operators $M_{z^{\lambda}}M_{z^{\lambda}}^{*}$ and $M_{z^{\mu}}M_{z^{\mu}}^{*}$
commute.\end{lem}
\begin{proof}
For $\alpha$ in $\mathbb{N}_{0}^{d}$,
\[
z^{\lambda}\left(\partial^{\alpha}z^{\mu}\right)=c_{\mu}z^{\lambda_{1}+\mu_{1}-\alpha_{1}}\cdots z^{\lambda_{d}+\mu_{d}-\alpha_{d}},
\]
where
\[
c_{\mu}=\begin{cases}
{\displaystyle \prod_{i=1}^{d}\mu_{i}\left(\mu_{i}-1\right)\cdots\left(\mu_{i}-\alpha_{i}+1\right)} & \mbox{if}\ \alpha_{i}\leq\mu_{i}\ \mbox{for}\ 1\leq i\leq d,\\
0 & \mbox{otherwise}.
\end{cases}
\]
Similarly,
\[
\left(\partial^{\alpha}z^{\lambda}\right)z^{\mu}=c_{\lambda}z^{\lambda_{1}+\mu_{1}-\alpha_{1}}\cdots z^{\lambda_{d}+\mu_{d}-\alpha_{d}},
\]
where
\[
c_{\lambda}=\begin{cases}
{\displaystyle \prod_{i=1}^{d}\lambda_{i}\left(\lambda_{i}-1\right)\cdots\left(\lambda_{i}-\alpha_{i}+1\right)} & \mbox{if}\ \alpha_{i}\leq\lambda_{i}\ \mbox{for}\ 1\leq i\leq d,\\
0 & \mbox{otherwise}.
\end{cases}
\]
Let $\nu=\left(\lambda_{1}+\mu_{1}-\alpha_{1},\ldots,\lambda_{d}+\mu_{d}-\alpha_{d}\right)$.
Then 
\[
M_{z^{\lambda}}M_{\partial^{\alpha}z^{\mu}}M_{z^{\mu}}^{*}M_{\partial^{\alpha}z^{\lambda}}^{*}=\begin{cases}
c_{\lambda}c_{\mu}M_{z^{\nu}}M_{z^{\nu}}^{*} & \mbox{if}\ \nu\in\mathbb{N}_{0}^{d},\\
0 & \mbox{otherwise}.
\end{cases}
\]
In particular, this operator is self-adjoint. Therefore, by Lemma
\ref{lem:simple-comm-criterion}, the operators $M_{z^{\lambda}}M_{z^{\lambda}}^{*}$
and $M_{z^{\mu}}M_{z^{\mu}}^{*}$ commute.
\end{proof}

\begin{lem}
\label{lem:disjoint-comm}Let $p$ and $q$ be homogeneous in $\mathbb{C}[z]$
in distinct variables. Then the operators $M_{p}M_{p}^{*}$ and $M_{q}M_{q}^{*}$
commute.\end{lem}
\begin{proof}
Since $p$ and $q$ are polynomials in disjoint variables, for every
$\alpha\in\mathbb{N}_{0}^{d}\backslash\left\{ 0\right\} $, at least
one of $\partial^{\alpha}p$ and $\partial^{\alpha}q$ must be zero,
and hence at least one of $M_{\partial^{\alpha}p}$ and $M_{\partial^{\alpha}q}$
must be zero. In particular, this implies that $M_{p}M_{\partial^{\alpha}q}M_{q}^{*}M_{\partial^{\alpha}p}^{*}=0$,
and it follows from Lemma \ref{lem:simple-comm-criterion} that $M_{p}M_{p}^{*}$
and $M_{q}M_{q}^{*}$ commute.
\end{proof}

\subsection{\label{sub:perp-submodules-rank-one}Perpendicular submodules}

In this section, we will establish the perpendicularity of many families
of submodules of $H_{d}^{2}$ using the criteria from Section \ref{sub:criteria-for-perpendicularity}.

\begin{prop}
\label{prop:linear-family-perpendicular-criterion}Let $N_{1},\ldots,N_{n}$
be submodules of $H_{d}^{2}$ that are generated by mutually orthogonal
sets of linear polynomials $F_{1},\ldots,F_{n}$ respectively. Then
the family $\{N_{1},\ldots,N_{n}\}$ is perpendicular.\end{prop}
\begin{proof}
This follows immediately from Lemma \ref{lem:generators-comm-decomp}
and Lemma \ref{eq:lem-linear-comm-criterion-1}.
\end{proof}

\begin{prop}
\label{prop:monomial-family-perpendicular}Let $N_{1},\ldots,N_{n}$
be submodules of $H_{d}^{2}$ each generated by monomials. Then \textup{the
family $\{N_{1},\ldots,N_{n}\}$} is perpendicular.\end{prop}
\begin{proof}
This follows immediately from Lemma \ref{lem:generators-comm-decomp}
and Lemma \ref{lem:monomials-comm}.
\end{proof}

\begin{prop}
\label{prop:disjoint-vars-family-perpendicular}Let $N_{1},\ldots,N_{n}$
be submodules of $H_{d}^{2}$ generated by sets of homogeneous polynomials
$F_{1},\ldots,F_{n}$ respectively. Suppose that there are disjoint
subsets $Z_{1},\ldots,Z_{n}$ of $\left\{ z_{1},\ldots,z_{d}\right\} $
such that
\[
F_{i}\subseteq\mathbb{C}\left[Z_{i}\right],\quad1\leq i\leq n.
\]
Then the family $\left\{ N_{1},\ldots,N_{n}\right\} $ is perpendicular.\end{prop}
\begin{proof}
This follows immediately from Lemma \ref{lem:generators-comm-decomp}
and Lemma \ref{lem:disjoint-comm}.
\end{proof}

We can strengthen Proposition \ref{prop:disjoint-vars-family-perpendicular}
using results of Carlini and Reznick. The following result is Lemma
3.1 in \cite{Rez93}.
\begin{prop}
[Rez93]Let $p_{1},\ldots,p_{n}$ be homogeneous polynomials in $\mathbb{C}[z]$.
If the sets
\[
\left\{ \partial_{z^{\alpha}}p_{i}\mid\left|\alpha\right|=\deg\left(p_{i}\right)-1,\ \alpha\in\mathbb{N}_{0}^{d}\right\} ,\quad1\leq i\leq n,
\]
are mutually orthogonal, then there is a unitary change of variables
such that the polynomials $p_{1},\ldots,p_{n}$ are polynomials in
disjoint variables.
\end{prop}

The following result is Proposition 1 in \cite{Car06}.
\begin{prop}
[Car06]Let $p_{1},\ldots,p_{n}$ be homogeneous polynomials in $\mathbb{C}[z]$.
If the sets
\[
\left\{ \nabla p_{i}\left(z\right)\mid z\in\mathbb{C}^{d}\right\} ,\quad1\leq i\leq n,
\]
are mutually orthogonal, where $\nabla p$ denotes the gradient of
$p$, then there is a unitary change of variables such that the polynomials
$p_{1},\ldots,p_{n}$ are polynomials in disjoint variables.
\end{prop}

We immediately obtain the following two results.
\begin{prop}
\label{prop:disjoint-vars-family-perpendicular-1}Let $N_{1},\ldots,N_{n}$
be submodules of $H_{d}^{2}$ generated by sets of homogeneous polynomials
$F_{1},\ldots,F_{n}$ respectively. If the sets 
\[
\left\{ \partial^{\alpha}\left(p\right)\mid\left|\alpha\right|=\deg\left(p\right)-1,\ \alpha\in\mathbb{N}_{0}^{d},\ p\in F_{i}\right\} ,\quad1\leq i\leq n,
\]
are mutually orthogonal, then the family $\left\{ N_{1},\ldots,N_{n}\right\} $
is perpendicular.
\end{prop}

\begin{prop}
\label{prop:disjoint-vars-family-perpendicular-2}Let $N_{1},\ldots,N_{n}$
be submodules of $H_{d}^{2}$ generated by sets of homogeneous polynomials
$F_{1},\ldots,F_{n}$ respectively. If the sets
\[
\left\{ \left(\nabla p\right)\left(\lambda\right)\mid\lambda\in\mathbb{C}^{d},\ p\in F_{i}\right\} ,\quad1\leq i\leq n,
\]
are mutually orthogonal, then the family $\left\{ N_{1},\ldots,N_{n}\right\} $
is perpendicular.
\end{prop}

\subsection{Perpendicularity and tensor products}

The results obtained in Section \ref{sub:criteria-for-perpendicularity}
and Section \ref{sub:perp-submodules-rank-one} only apply to submodules
of $H_{d}^{2}$. Because we also need to consider higher-rank submodules
of $rH_{d}^{2}$, in this section we consider tensor products of perpendicular
submodules. 
\begin{thm}
\label{thm:angle-tensored-perpendicular-submodules}Let $\left\{ N_{1},\ldots,N_{n}\right\} $
be a perpendicular family of submodules of $H_{d}^{2}$, and let $V_{1},\ldots V_{n}$
be arbitrary subspaces of $\mathbb{C}^{r}$. Let $M_{1},\ldots,M_{n}$
denote the $rH_{d}^{2}$ submodules
\[
M_{i}=N_{i}\otimes V_{i},\quad1\leq i\leq n.
\]
Then the algebraic sum $M_{1}+\ldots+M_{n}$ is closed.
\end{thm}

\begin{proof}
Let $E_{1},\ldots,E_{n}$ denote the projections onto $V_{1},\ldots,V_{n}$
respectively. Then it's clear that
\[
P_{M_{i}}=P_{N_{i}}\otimes E_{i},\quad1\leq i\leq n.
\]
We will prove that the operator 
\[
P_{M_{1}}+\ldots+P_{M_{n}}
\]
has closed range. By Lemma \ref{lem:vec-sum-closed-iff-proj-sum-closed},
this will imply the desired result.

We proceed by induction on $n$, the size of the family $\left\{ N_{1},\ldots,N_{n}\right\} $.
For $n=1$, the result is trivially true. Therefore, suppose that
$n\geq2$ with the induction hypothesis that the result is true for
a perpendicular family of submodules of $H_{d}^{2}$ if the size of
the family is at most $n-1$.

Let $Q_{0},\ldots,Q_{n}$ denote the projections onto $rH_{d}^{2}$
defined by 
\begin{eqnarray*}
Q_{0} & = & P_{N_{1}}\cdots P_{N_{n}}\otimes I\\
Q_{1} & = & P_{N_{1}}^{\perp}\otimes I\\
Q_{2} & = & P_{N_{1}}P_{N_{2}}^{\perp}\otimes I\\
 & \vdots\\
Q_{n} & = & P_{N_{1}}\cdots P_{N_{n-1}}P_{N_{n}}^{\perp}\otimes I.
\end{eqnarray*}
Since the family $\left\{ N_{1},\ldots,N_{n}\right\} $ is perpendicular,
Proposition \ref{prop:perpendicular-iff-projections-commuting} implies
that the projections $P_{N_{1}},\ldots,P_{N_{n}}$ commute, and hence
that the projections $Q_{0},\ldots,Q_{n}$ also commute. It's also
clear that 
\begin{equation}
Q_{i}P_{M_{j}}=P_{M_{j}}Q_{i},\quad1\leq i,j\leq n.\label{eq:prop-angle-tensored-perpendicular-submodules-1}
\end{equation}
Furthermore, by construction the projections $Q_{0},\ldots,Q_{n}$
are orthogonal,
\begin{equation}
Q_{i}Q_{j}=0,\quad0\leq i<j\leq n,\label{eq:prop-angle-tensored-perpendicular-submodules-2}
\end{equation}
and we can decompose the identity operator on $rH_{d}^{2}$ as
\begin{equation}
I=Q_{0}+Q_{1}+Q_{2}+\ldots+Q_{n}.\label{eq:prop-angle-tensored-perpendicular-submodules-3}
\end{equation}
Therefore, by (\ref{eq:prop-angle-tensored-perpendicular-submodules-1}),
(\ref{eq:prop-angle-tensored-perpendicular-submodules-2}) and (\ref{eq:prop-angle-tensored-perpendicular-submodules-3}),
we can write
\begin{eqnarray}
P_{M_{1}}+\ldots+P_{M_{n}} & = & \left(\sum_{i=0}^{n}Q_{i}\right)\left(\sum_{j=1}^{n}P_{M_{j}}\right)\left(\sum_{i=0}^{m}Q_{i}\right)\label{eq:prop-angle-tensored-perpendicular-submodules-4}\\
 & = & \sum_{i=0}^{n}\sum_{j=0}^{n}Q_{i}P_{M_{j}}.\nonumber 
\end{eqnarray}
Now, for $1\leq j\leq n$,
\begin{align*}
Q_{0}P_{M_{j}} & =\left(P_{N_{1}}\cdots P_{N_{n}}\otimes I\right)\left(P_{N_{j}}\otimes E_{j}\right)\\
 & =P_{N_{1}}\cdots P_{N_{n}}\otimes E_{j}\\
 & =Q_{0}\left(I\otimes E_{j}\right),
\end{align*}
and this gives
\begin{equation}
Q_{0}\left(\sum_{j=1}^{n}P_{M_{j}}\right)=Q_{0}\sum_{j=1}^{n}\left(I\otimes E_{j}\right).\label{eq:prop-angle-tensored-perpendicular-submodules-5}
\end{equation}
By a similar calculation, for $1\leq i\leq n$,
\begin{align*}
Q_{i}P_{M_{i}} & =\left(P_{N_{1}}\cdots P_{N_{i-1}}P_{N_{i}}^{\perp}\otimes I\right)\left(P_{N_{i}}\otimes E_{i}\right)\\
 & =P_{N_{1}}\cdots P_{N_{i-1}}P_{N_{i}}^{\perp}P_{N_{i}}\otimes E_{i}\\
 & =0,
\end{align*}
and this gives 
\begin{equation}
Q_{i}\sum_{j=1}^{n}P_{M_{j}}=Q_{i}\sum_{\substack{j=1\\
j\ne i
}
}^{n}P_{M_{j}}.\label{eq:prop-angle-tensored-perpendicular-submodules-6}
\end{equation}
Let $X_{0}$ denote the operator on $rH_{d}^{2}$ defined by 
\[
X_{0}=I\otimes\sum_{j=1}^{n}E_{j},
\]
and let $X_{1},\ldots,X_{n}$ denote the operators on $H_{d}^{2}\otimes\mathbb{C}^{r}$
defined by 
\[
X_{i}=\sum_{\substack{j=1\\
j\ne i
}
}^{n}P_{M_{j}},\quad1\leq i\leq n.
\]
Then by the induction hypothesis and Lemma \ref{lem:vec-sum-closed-iff-proj-sum-closed},
the operators $X_{0},\ldots,X_{n}$ each have closed range, and by
(\ref{eq:prop-angle-tensored-perpendicular-submodules-4}), (\ref{eq:prop-angle-tensored-perpendicular-submodules-5})
and (\ref{eq:prop-angle-tensored-perpendicular-submodules-6}), we
can write
\[
P_{M_{1}}+\ldots+P_{M_{n}}=Q_{0}X_{0}+\ldots+Q_{n}X_{n}.
\]
Since the operators $X_{0},\ldots,X_{n}$ commute with the projections
$Q_{0},\ldots,Q_{n}$, it follows that the operators $Q_{0}X_{0},\ldots,Q_{n}X_{n}$
each have closed range. Therefore, since the projections $Q_{0},\ldots,Q_{n}$
are orthogonal, this means that we have written $P_{M_{1}}+\ldots+P_{M_{n}}$
as the direct sum of $n+1$ operators that each have closed range.
It follows that the range of this operator is also closed.
\end{proof}

\section{Essential normality}

\subsection{Essential normality and perpendicularity}

\begin{lem}
\label{lem:tensor-p-ess-norm-is-p-ess-norm}Let $N$ be a $p$-essentially
normal submodule of $H_{d}^{2}$ for $p>d$, and let $V$ be an arbitrary
subspace of $\mathbb{C}^{r}$. Then the $rH_{d}^{2}$ submodule $N\otimes V$
is also $p$-essentially normal. \end{lem}
\begin{proof}
Let $E$ denote the projection onto $V$, and let $M=N\otimes V$.
Then it's clear that
\[
P_{M}=P_{N}\otimes E.
\]
By Theorem \ref{thm:arveson-ess-norm-equiv-conditions}, the $p$-essential
normality of $N$ implies that the projection $P_{N}$ $2p$-essentially
commutes with the coordinate multiplication operators $M_{z_{1}},\ldots,M_{z_{d}}$
on $H_{d}^{2}$, i.e.
\[
M_{z_{i}}P_{N}-P_{N}M_{z_{i}}\in\mathcal{L}^{2p},\quad1\leq i\leq d,
\]
where $\mathcal{L}^{2p}$ denotes the set of Schatten $2p$-class
operators on $H_{d}^{2}$. Recall from Section \ref{sub:finite-rank-d-shift}
that we can write the coordinate multiplication operators on $rH_{d}^{2}$
as $M_{z_{1}}\otimes I,\ldots,M_{z_{d}}\otimes I$. Hence for $1\leq i\leq d$,
\begin{align*}
\left(M_{z_{i}}\otimes I\right)P_{M}-P_{M}\left(M_{z_{i}}\otimes I\right) & =\left(M_{z_{i}}\otimes I\right)\left(P_{N}\otimes E\right)-\left(P_{N}\otimes E\right)\left(M_{z_{i}}\otimes I\right)\\
 & =\left(M_{z_{i}}P_{N}-P_{N}M_{z_{i}}\right)\otimes E\in\mathcal{L}^{2p},
\end{align*}
since $E$ is a finite rank projection. Therefore, by Theorem \ref{thm:arveson-ess-norm-equiv-conditions},
$M$ is $2p$-essentially normal.
\end{proof}

\begin{thm}
\label{thm:ess-norm-of-perp-submodules}Let $\{N_{1},\ldots,N_{n}\}$
be a perpendicular family of $p$-essentially decomposable submodules
of $H_{d}^{2}$ for $p>d$, and let $V_{1},\ldots V_{n}$ be arbitrary
subspaces of $\mathbb{C}^{r}$. Let $M_{1},\ldots,M_{n}$ denote the
$rH_{d}^{2}$ submodules
\[
M_{i}=N_{i}\otimes V_{i},\quad1\leq i\leq n,
\]
Then the $rH_{d}^{2}$ submodule $\overline{M_{1}+\ldots+M_{n}}$
is also $p$-essentially normal.\end{thm}
\begin{proof}
By Lemma \ref{lem:tensor-p-ess-norm-is-p-ess-norm}, the submodules
$M_{1},\ldots,M_{n}$ are $p$-essentially normal. Let $M$ denote
the $rH_{d}^{2}$ submodule $M=\overline{M_{1}+\ldots+M_{n}}$. By
Theorem \ref{thm:angle-tensored-perpendicular-submodules}, we can
write $M=M_{1}+\ldots+M_{n}.$ It follows that $M$ is $p$-essentially
decomposable, and hence by Theorem \ref{thm:ess-decomp-implies-ess-norm}
that $M$ is $p$-essentially normal.
\end{proof}

\subsection{\label{sub:ess-norm-certain-submodules}Essential normality}

In this section, we establish our main results on the essential normality
of homogeneous submodules of $rH_{d}^{2}$. We will require Guo and
Wang's result, Theorem 2.2 from \cite{GW08}, about the essential
normality of singly generated homogeneous submodules.
\begin{thm}
[Guo-Wang]\label{thm:guo-wang-principal-submodules-ess-norm}Every
submodule of $rH_{d}^{2}$ generated by a single homogeneous polynomial
is $p$-essentially normal for every $p>d$.
\end{thm}

The next result is well known. It was proved, for example, by Shalit
in \cite{Sha11}, using his results on stable division. The methods
introduced here provide a new proof.
\begin{thm}
\label{thm:submodule-lin-polys-ess-norm} Every submodule of $H_{d}^{2}$
generated by linear polynomials is $p$-essentially normal for every
$p>d$. \end{thm}
\begin{proof}
Let $N$ be a submodule of $H_{d}^{2}$ generated by linear polynomials
$p_{1},\ldots,p_{n}$ in $\mathbb{C}[z]$ . By applying the Gram-Schmidt
process if necessary, we can assume that the set $\{p_{1},\dots,p_{n}\}$
is orthogonal in $H_{d}^{2}$. Let $N_{1},\ldots,N_{n}$ denote the
$H_{d}^{2}$ submodules generated by $p_{1},\ldots,p_{n}$ respectively.
Then Theorem \ref{thm:guo-wang-principal-submodules-ess-norm} implies
that these submodules are each $p$-essentially normal for every $p>d$,
and Proposition \ref{prop:linear-family-perpendicular-criterion}
implies that the family $\{N_{1},\ldots,N_{n}\}$ is perpendicular.
Note that $N=\overline{N_{1}+\ldots+N_{n}}$. Hence by Theorem  \ref{thm:ess-norm-of-perp-submodules},
$N$ is also $p$-essentially normal for every $p>d$. 
\end{proof}

The next result is new. It establishes the essential normality of
submodules of $rH_{d}^{2}$ that are generated by certain linear polynomials.
We note Arveson's result from \cite{Arv07} that the problem of the
essential normality of homogeneous submodules of $rH_{d}^{2}$ is
equivalent to the problem of the essential normality of submodules
of $rH_{d}^{2}$ generated by arbitrary linear polynomials. 
\begin{thm}
\label{thm:ortho-linear-poly-ess-decomp}Let $F_{1},\ldots,F_{n}$
be mutually orthogonal sets of linear polynomials, and let $X_{1},\ldots,X_{n}$
be arbitrary sets of vectors in $\mathbb{C}^{r}$. Then the $rH_{d}^{2}$
submodule generated by the set of vector-valued polynomials
\begin{equation}
\{p\otimes\xi\mid p\in F_{i},\ \xi\in X_{i},\ 1\leq i\leq n\}\label{eq:thm-ortho-linear-poly-ess-decomp-1}
\end{equation}
is $p$-essentially normal for every $p>d$.\end{thm}
\begin{proof}
Let $N_{1},\ldots,N_{n}$ denote the $H_{d}^{2}$ submodules generated
by $F_{1},\ldots,F_{n}$ respectively, and let $V_{1},\ldots,V_{n}$
denote the $\mathbb{C}^{r}$ subspaces spanned by $X_{1},\ldots,X_{n}$
respectively. Let $M$ denote the $rH_{d}^{2}$ submodule generated
by the set (\ref{eq:thm-ortho-linear-poly-ess-decomp-1}), and let
$M_{1},\ldots,M_{n}$ denote the $rH_{d}^{2}$ submodules
\[
M_{i}=N_{i}\otimes V_{i},\quad1\leq i\leq n.
\]
Then Theorem \ref{thm:submodule-lin-polys-ess-norm} implies that
each of the submodules $N_{1},\ldots,N_{n}$ is $p$-essentially normal
for every $p>d$, and Proposition \ref{prop:linear-family-perpendicular-criterion}
implies that the family $\left\{ N_{1},\ldots,N_{n}\right\} $ is
perpendicular. Note that $M=\overline{M_{1}+\ldots+M_{n}}$. Hence
by Theorem \ref{thm:ess-norm-of-perp-submodules}, $M$ is also $p$-essentially
normal for every $p>d$. 
\end{proof}

The next theorem is Arveson's main result from \cite{Arv05}. Shalit
also gave a proof of this result in \cite{Sha11} using his results
on stable division. The methods introduced here provide a new and
simple proof. Recall that a monomial of $rH_{d}^{2}$ is an element
in $r\mathbb{C}[z]$ of the form $z^{\alpha}\otimes\xi$ for some
$\alpha$ in $\mathbb{N}_{0}^{d}$ and $\xi$ in $\mathbb{C}^{r}$. 
\begin{thm}
[Arveson]\label{thm:submodule-monomials-essdecomp}Every submodule
of $rH_{d}^{2}$ generated by monomials is $p$-essentially normal
for every $p>d$.\end{thm}
\begin{proof}
Let $N$ be a submodule of $rH_{d}^{2}$ generated by monomials, say
$z^{\alpha_{1}}\otimes\xi_{1},\ldots,z^{\alpha_{n}}\otimes\xi_{n}$
in $r\mathbb{C}[z]$ for $\alpha_{1},\ldots,\alpha_{n}$ in $\mathbb{N}_{0}^{d}$.
Let $N_{1},\ldots,N_{n}$ denote the $rH_{d}^{2}$ submodules generated
by $z^{\alpha_{1}},\ldots,z^{\alpha_{n}}$ respectively, and let $V_{1},\ldots,V_{n}$
denote the one-dimensional subspaces of $\mathbb{C}^{r}$ spanned
by $\xi_{1},\ldots,\xi_{n}$ respectively. Then Theorem \ref{thm:guo-wang-principal-submodules-ess-norm}
implies that the submodules $N_{1},\ldots,N_{n}$ are $p$-essentially
normal for every $p>d$, and Proposition \ref{prop:monomial-family-perpendicular}
implies that the family $\{N_{1},\ldots,N_{n}\}$ is perpendicular.
Note that $N=\overline{N_{1}+\ldots+N_{n}}$. Hence by Theorem \ref{thm:ess-norm-of-perp-submodules},
$N$ is also $p$-essentially normal for every $p>d$.
\end{proof}

Recall Shalit's result from \cite{Sha11} that a submodule generated
by polynomials in two variables has the stable division property.
Shalit used this result to prove the next theorem that these submodules
are essentially normal. However, starting from Proposition \ref{prop:grobner-basis-implies-friedrichs-num},
we can also view Shalit's proof of the stable division property for
these submodules as a method for establishing essential decomposability.
In this case, the methods introduced here provide a new proof. 
\begin{thm}
[Shalit]\label{thm:shalit-ess-norm}Let $F$ be a set of homogeneous
polynomials. Suppose that there is a subset $Z$ of $\{z_{1},\ldots,z_{d}\}$,
of size at most $2$, such that $F\subseteq\mathbb{C}\left[Z\right]$.
Then the $H_{d}^{2}$ submodule generated by $F$ is $p$-essentially
normal for every $p>d$.\end{thm}
\begin{proof}
Let $N$ denote the $H_{d}^{2}$ submodule generated by $F$, and
let $\{p_{1},\ldots,p_{n}\}$ be a Groebner basis consisting of homogeneous
polynomials that generates $N$. Let $N_{1},\ldots,N_{n}$ denote
the $H_{d}^{2}$ submodules generated by $p_{1},\ldots,p_{n}$ respectively.
Then by Theorem \ref{thm:guo-wang-principal-submodules-ess-norm},
$N_{1},\ldots,N_{n}$ are $p$-essentially normal for every $p>d$.
Note that the polynomials $p_{1},\ldots,p_{n}$ belong to $\mathbb{C}[Z]$.
Hence by Proposition \ref{prop:grobner-basis-implies-friedrichs-num},
$N=N_{1}+\ldots+N_{n}$, and $N$ is $p$-essentially decomposable
for every $p>d$. It follows from Theorem \ref{thm:ess-decomp-implies-ess-norm}
that $M$ is $p$-essentially normal for every $p>d$.
\end{proof}

The next result is new. It implies the essential normality of a large
new class of submodules of $rH_{d}^{2}$.
\begin{thm}
\label{thm:disjoint-vars-ess-decomp}Let $F_{1},\ldots,F_{n}$ be
sets of homogeneous polynomials that each generate $p$-essentially
normal submodules of $H_{d}^{2}$ for $p>d$. Suppose that there are
disjoint subsets $Z_{1},\ldots,Z_{n}$ of $\{z_{1},\ldots,z_{d}\}$
such that
\[
F_{i}\subseteq\mathbb{C}[Z_{i}],\quad1\leq i\leq n.
\]
 Let $X_{1},\ldots,X_{n}$ be arbitrary sets of vectors in $\mathbb{C}^{r}$.
Then the $rH_{d}^{2}$ submodule generated by the set of vector-valued
polynomials
\begin{equation}
\{p\otimes\xi\mid p\in F_{i},\ \xi\in X_{i},\ 1\leq i\leq n\}\label{eq:thm-disjoint-vars-ess-decomp-1}
\end{equation}
is $p$-essentially normal.\end{thm}
\begin{proof}
Let $N_{1},\ldots,N_{n}$ denote the $H_{d}^{2}$ submodules generated
by $F_{1},\ldots,F_{n}$ respectively, and let $V_{1},\ldots,V_{n}$
denote the $\mathbb{C}^{r}$ subspaces spanned by $X_{1},\ldots,X_{n}$
respectively. Let $M$ denote the $rH_{d}^{2}$ submodule generated
by the set (\ref{eq:thm-disjoint-vars-ess-decomp-1}), and let $M_{1},\ldots,M_{n}$
denote the $rH_{d}^{2}$ submodules
\[
M_{i}=N_{i}\otimes V_{i},\quad1\leq i\leq n.
\]
The submodules $N_{1},\ldots,N_{n}$ are $p$-essentially normal by
assumption, and Proposition \ref{prop:disjoint-vars-family-perpendicular}
implies that the family $\{N_{1},\ldots,N_{n}\}$ is perpendicular.
Note that $M=\overline{M_{1}+\ldots+M_{n}}$. Hence by Theorem \ref{thm:ess-norm-of-perp-submodules},
$M$ is also $p$-essentially normal.
\end{proof}

Replacing the use of Proposition \ref{prop:disjoint-vars-family-perpendicular}
in the proof of Theorem \ref{thm:disjoint-vars-ess-decomp} with Proposition
\ref{prop:disjoint-vars-family-perpendicular-1} and Proposition \ref{prop:disjoint-vars-family-perpendicular-2}
respectively, we immediately obtain the following strengthened results.

\begin{thm}
Let $F_{1},\ldots,F_{n}$ be sets of homogeneous polynomials that
each generate $p$-essentially normal submodules of $H_{d}^{2}$ for
$p>d$. Suppose that the sets
\[
\left\{ \partial^{\alpha}\left(p\right)\mid\left|\alpha\right|=\deg\left(p\right)-1,\ \alpha\in\mathbb{N}_{0}^{d},\ p\in F_{i}\right\} ,\quad1\leq i\leq n,
\]
are mutually orthogonal. Let $X_{1},\ldots,X_{n}$ be arbitrary sets
of vectors in $\mathbb{C}^{r}$. Then the $rH_{d}^{2}$ submodule
generated by the set of vector-valued polynomials
\[
\left\{ p\otimes\xi\mid p\in F_{i},\ \xi\in X_{i},\ 1\leq i\leq n\right\} 
\]
is $p$-essentially normal.
\end{thm}

\begin{thm}
Let $F_{1},\ldots,F_{n}$ be sets of homogeneous polynomials that
each generate $p$-essentially normal submodules of $H_{d}^{2}$ for
$p>d$. Suppose that the sets
\[
\left\{ \left(\nabla p\right)\left(\lambda\right)\mid\lambda\in\mathbb{C}^{d},\ p\in F_{i}\right\} ,\quad1\leq i\leq n,
\]
are mutually orthogonal. Let $X_{1},\ldots,X_{n}$ be arbitrary sets
of vectors in $\mathbb{C}^{r}$. Then the $rH_{d}^{2}$ submodule
generated by the set of vector-valued polynomials
\[
\left\{ p\otimes\xi\mid p\in F_{i},\ \xi\in X_{i},\ 1\leq i\leq n\right\} 
\]
is $p$-essentially normal.
\end{thm}

The next theorem follows from a combination of the results above.
\begin{thm}
\label{thm:nicest-thm}Let $F_{1},\ldots,F_{n}$ be sets of homogeneous
polynomials in $\mathbb{C}[z]$. Suppose that there are disjoint subsets
$Z_{1},\ldots,Z_{n}$ of $\left\{ z_{1},\ldots,z_{d}\right\} $, each
of size at most $2$, such that
\[
F_{i}\subseteq\mathbb{C}\left[Z_{i}\right],\quad1\leq i\leq n.
\]
Let $X_{1},\ldots,X_{n}$ be arbitrary sets of vectors in $\mathbb{C}^{r}$.
Then the $H_{d}^{2}\otimes\mathbb{C}^{r}$ submodule generated by
the set of vector-valued polynomials.
\[
\left\{ p\otimes\xi\mid p\in F_{i},\ \xi\in X_{i},\ 1\leq i\leq n\right\} 
\]
is $p$-essentially normal for every $p>d$.\end{thm}
\begin{proof}
This follows immediately from Theorem \ref{thm:shalit-ess-norm} and
Theorem \ref{thm:disjoint-vars-ess-decomp}. 
\end{proof}

\begin{example}
For every even $d\geq1$ and $n\geq1$, let $N$ denote the $H_{d}^{2}$
submodule generated by the set of polynomials $p_{1},\ldots,p_{d/2}$,
where
\begin{align*}
p_{1}\left(z_{1},\ldots,z_{d}\right) & =z_{1}^{n}+z_{1}^{n-1}z_{2}+\ldots+z_{1}z_{2}^{n-1}+z_{2}^{n}\\
 & \ \vdots\\
p_{d/2}\left(z_{1},\ldots,z_{d}\right) & =z_{d-1}^{n}+z_{d-1}^{n-1}z_{d}+\ldots+z_{d-1}z_{d}^{n-1}+z_{d}^{n}
\end{align*}
\[
.
\]
 Then Theorem \ref{thm:nicest-thm} implies that $N$ is $p$-essentially
normal for every $p>d$.
\end{example}

\section*{Acknowledgements}

The author is grateful to Ken Davidson and Orr Shalit for many helpful
comments and suggestions. The author would also like to thank the
anonymous referee for their suggestions, which greatly improved the
exposition.

\end{document}